\documentclass[]{publmathdeb}

\usepackage{amsmath,amsfonts,amssymb}

\usepackage{mathtools}		
\usepackage{float}			

\newtheorem{theorem}{Theorem}
\newtheorem{lemma}{Lemma}

\newcommand{\R}{\mathbb{R}}
\newcommand{\Q}{\mathbb{Q}}
\newcommand{\Z}{\mathbb{Z}}
\newcommand{\N}{\mathbb{N}}

\DeclarePairedDelimiter\autoset{\{}{\}}
\newcommand{\set}[1]{\autoset*{#1}}

\DeclarePairedDelimiter\absolute{|}{|}
\newcommand{\abs}[1]{\absolute*{#1}}

\DeclarePairedDelimiter\autobracket{(}{)}
\newcommand{\br}[1]{\autobracket*{#1}}

\author{Tobias Hilgart}
\address{Department of Mathematics\\
	University of Salzburg\\
	A-5020 Salzburg, Hellbrunnerstraße 34\\
	Austria}
\email{tobias.hilgart@stud.sbg.ac.at}

\title[Cubic split Thue equations parametrised by linear recurrence sequences]{On families of cubic split Thue equations parametrised by linear recurrence sequences}

\begin{document}
	
	\begin{abstract}
		Let $\br{A_n}_{n\in \N}, \br{B_n}_{n\in \N} \in \Z^{\N}$ be two linear-recurrent sequences that meet a dominant root condition and a few more technical requirements. We show that the split family of Thue equations
		\[
		\abs{X\br{X-A_n Y}\br{X-B_n Y} - Y^3} = 1
		\]
		has but the trivial solutions $\pm\set{ \br{0,1}, \br{1,0}, \br{A_n,1}, \br{B_n,1} }$, if the parameter $n$ is larger than some effectively computable constant.
	\end{abstract}
	
	\maketitle
	
	\section{Introduction}
	
		Thue equations are Diophantine equations $F(X,Y) = m$, where $F\in\Z[X,Y]$ is an irreducible form of degree at least $3$ and $m\neq 0$ a fixed integer. A. Thue showed that such equations only have finitely many solutions \cite{Thu1909}. But the methods he used were ineffective and not suitable to derive bounds on the number of solutions for concrete equations. It was Baker \cite{Bak1968} that could give such effective bounds using his theory on linear forms in logarithms. These bounds have been used and refined many times since. \par\medskip	
		
		E. Thomas studied so-called split Thue equations,
		\[
		X\prod_{i=2}^N \br{ X-p_i(a)Y } - Y^N = \pm 1,
		\]
		that are parametrised by monic polynomials $p_i \in \Z[a]$. He conjectured that, given some technical conditions for the $p_i$, the equation would only have trivial solutions immediately apparent by its form, provided the parameter $a \geq a_0$ is sufficiently large. The bound $a_0$ would also be effectively computable. The conjecture has since been proven, first by Thomas \cite{Tho1990} in the cubic and Heuberger \cite{Heu2001} in the general case. \par\medskip
		
		The methods involved rely heavily on Baker's theory of linear forms in logarithms but not so much on the parameter functions being polynomials. Hence, the authors in \cite{Hil2022} and \cite{Vuk2021} made attempts using other parameter functions, namely linear-recurrent sequences. This paper, mimicking Thomas' work on polynomial-parametrised split Thue equations, strives to proof the cubic case for when the parameter functions are linear-recurrent sequences in as much of a general setting as possible. \par\medskip
		
		Let $\br{A_n}_{n_\in \N}, \br{B_n}_{n\in\N}$ be two integer sequences that satisfy some linear recurrences with some starting conditions. We denote the characteristic polynomials of the recurrence sequences as 
		\[
		\br{X-\alpha}^{d_A}\prod_{i=1}^{m_A}  \br{ X-\alpha_i }^{d_{A,i}}, \text{ and } \br{X-\beta}^{d_B}\prod_{i=1}^{m_B}  \br{ X-\beta_i }^{d_{B,i}}
		\]
		respectively to give the algebraic multiplicity of each characteristic root. We will consider the sequences almost exclusively via their explicit formulae
		\begin{equation}\label{eq:A_n def}
			A_n = c_A(n) \alpha^n + c_{A,1}(n) \alpha_1^n + \cdots + c_{A, m_A}(n) \alpha_{m_A}^n \\
		\end{equation}
		\begin{equation}\label{eq:B_n def}
			B_n = c_B(n) \beta^n + c_{B,1}(n) \beta_1^n + \cdots + c_{B, m_B}(n) \alpha_{m_B}^n
		\end{equation}
		that always exist for linear recurrence sequences. \par\medskip
		
		Both are to admit to a dominant root condition, i.e. $\abs{\alpha} > \max_i \abs{\alpha_i}$ and $\abs{\beta} > \max_i \abs{\beta_i}$. We assume without loss of generality that $\abs{\alpha} \leq \abs{\beta}$. All characteristic roots must be algebraic over $\Q$. \par\medskip
		
		While we cannot expect the remaining terms, say $c_A$, to be constant, they are polynomials in the splitting field of the respective characteristic polynomial. In addition, the degree is bounded by the related multiplicity of the root, e.g. $\deg\br{c_A} \leq d_A - 1$. We also define
		\begin{align}\label{def:min-max-degrees}
			d_1 &:= \min\set{\deg\br{c_A}, \deg\br{c_{A,i}}, \deg\br{c_B}, \deg\br{c_{B,i}} }, \nonumber\\
			d_2 &:= \max \set{\deg\br{c_A}, \deg\br{c_{A,i}}, \deg\br{c_B}, \deg\br{c_{B,i}} }. 
		\end{align}
		
		In this setting, we prove the following result:
		
		\begin{theorem}\label{thm:main}
			The split family of Thue equations
			\begin{equation}\label{eq:thue}
				\abs{X(X-A_nY)(X-B_nY) - Y^3} = 1
			\end{equation}
			has only the trivial solutions $\set{ \pm \br{0,1}, \pm \br{A_n,1}, \pm\br{B_n,1}, \br{\pm 1, 0} }$ for sufficiently large $n$, if $A_n$ and $B_n$ meet a dominant root condition, where at least one dominant root has absolute value larger than $1$, and either
			\begin{itemize}
				\item $ 1 \leq A_n \leq B_n - 2$ or
				\item $ -1 \geq A_n \geq B_n + 3 $.
			\end{itemize}
			In case the dominant roots have equal absolute value, we also need that for sufficiently large $n$, both $\abs{c_B(n)} \neq \abs{c_A(n)}$ and one of the following conditions holds:
			\begin{itemize}
				\item $\abs{c_B(n)} > 1$ and $\abs{c_B(n)-c_A(n)} > \frac{1}{\abs{c_B(n)}}$
				\item $0 < \abs{c_B(n)} < 1$ and $0 < \abs{c_B(n)-c_A(n)} < 1$
				\item $\abs{c_B(n)} = 1$ and $\abs{c_B(n)-c_A(n)} \not\in\set{0,1}$
			\end{itemize}
		\end{theorem}
	
	\section{Preliminaries}
		Even if the terms $c_B(n), c_A(n)$ are polynomials in $n$ in general, we will forego to note the dependency every time for the sake of readability and simply write $c_B \; \br{= c_B(n)}$ and $c_A \; \br{ = c_A(n)}$. \par\medskip
		
		The conditions $1\leq A_n \leq B_n-2$ or $-1 \geq A_n \geq B_n + 3$ translate to $1\leq \abs{\alpha}, \abs{\beta}$, with $\abs{\alpha} \leq \abs{\beta}$ we further get that $\abs{\beta} > 1$. We need at least one of the dominant roots to be strictly larger than $1$ so that the corresponding sequence has exponential growth. This is not really a restriction, since $\abs{\alpha} = \abs{\beta} = 1$ would mean that both sequences have at most polynomial growth and that we are essentially in the case which Thomas had already studied and solved. \par\medskip
		
		Also note that in the case of $\abs{\alpha} = \abs{\beta}$, since $\abs{c_A} \neq \abs{c_B}$ at least for sufficiently large $n$, we can conclude that $\abs{c_A} \leq \abs{c_B}$, as neither of the inequalities $1\leq A_n \leq B_n-2$ or $-1\geq A_n \geq B_n + 3$ could hold otherwise. \par\medskip
		
		We first consider the case where a solution $(x,y)$ has small $\abs{y}$.
		
		\begin{lemma}
			If a solution $(x,y)$ to the Thue-Equation \eqref{eq:thue}, for sufficiently large $n$, has $y \in \set{-1,0,1}$, then it is a trivial solution.
		\end{lemma}
		\begin{proof}
			For $y = 0$, it follows immediately that $x^3 = \pm 1$, thus the only solutions with $y=0$ are trivial. If instead $y = \pm 1$, then either $X\br{X\pm A_n}\br{X\pm B_n} = 0$ or $X\br{X\pm A_n}\br{X\pm B_n} = \pm 2$.\par\medskip 
			
			The first case immediately gives $x \in \set{0, \mp A_n, \mp B_n}$ and thus that $(x,y)$ is a trivial solution. The second case is solvable for $x$ only if $\abs{B_n-A_n} \leq 2$, which can hold at most for small $n$: Since either $1\leq A_n \leq B_n -2$ or $-1\geq A_n \geq B_n + 3$, we also have $\abs{B_n-A_n}\geq 2$ and thus equality, which is impossible for $\abs{\alpha} < \abs{\beta}$ and precluded by $\abs{c_B} > \abs{c_A}$ for $\abs{\alpha}=\abs{\beta}$.
		\end{proof}
	
		\begin{lemma}\label{lem:log-ub}
			Assume $x \in \mathbb{C}$ such that $\abs{x-1} \leq 1/2$, then $\abs{\log\abs{x}} \leq 2 \abs{x-1}$.
		\end{lemma}
		\begin{proof}
			Use the Taylor series expansion of $\log x$ at $1$.
		\end{proof}
		
		\begin{theorem}[Thomas, \cite{Tho1979}]\label{thm:Thomas1979}
			Let $F_\lambda$ be the field generated by $\lambda$, where either
			\[
			\lambda\br{\lambda-r}\br{\lambda-s} + 1 = 0, \;\;\; 1 \leq r \leq s-3
			\]
			or
			\[
			\lambda\br{\lambda-r}\br{\lambda-s} - 1 = 0, \;\;\; 1 \leq r \leq s-2.
			\]
			Then $\set{\lambda, \lambda-r}$ is a fundamental system of units for the order generated by $\set{1,\lambda, \lambda^2}$.
		\end{theorem}
		
		\begin{theorem}[Bugeaud, Győry, \cite{BuGy1996}]\label{thm:bugy1996}
			Let $B \geq \max\set{\abs{m}, e}$, $\alpha$ be a root of $F(X,1)$, $K:=\Q(\alpha)$, $R:=R_K$ the regulator of $K$ and $r$ the unit rank of $K$. Let $H\geq 3$ be an upper bound for the absolute values of the coefficients of $F$ and $N$ its degree.
			
			Then all solutions $(x,y) \in \Z^2$ of the Thue equation $F(X,Y) = m$ satisfy
			\[
			\max\set{\log\abs{x},\log\abs{y}} \leq C(r,N) \cdot R \cdot \max\set{\log R, 1} \cdot \br{ R + \log(HB) },
			\]
			where $C(r,N) = 3^{r+27}(r+1)^{7r+19}N^{2N+6r+14}$.
		\end{theorem}
	
		For any algebraic number $\gamma$ of degree $d$ with conjugates $\gamma^{(1)}, \dots, \gamma^{(d)}$, we define the absolute logarithmic height $h$ of $\gamma$ as
		\[
		h\br{\gamma} = \frac{1}{d} \br{ \log\abs{a} + \sum_{i=1}^d \log \max\set{\abs{\gamma^{(i)}}, 1} },
		\]
		where $a$ is the leading coefficient of the minimal polynomial of $\gamma$ over $\Z$.
		
		\begin{theorem}[Baker, W\"ustholz, \cite{BaWh1993}]\label{thm:baker1993}
			Let $\gamma_1, \dots, \gamma_t$ be algebraic numbers, not $0$ or $1$, in $K= \Q(\gamma_1,\dots,\gamma_t)$ of degree $D$, let $b_1, \dots, b_t \in \Z$, and let
			\[
			\Lambda = b_1 \log\gamma_1 + \cdots + b_t \log\gamma_t
			\]
			be non-zero. Then
			\[
			\log\abs{\Lambda} \geq - 18(t+1)!t^{t+1}(32D)^{t+2}\log(2tD) h_1 \cdots h_t \log B,
			\]
			where 
			\[
			B \geq \max\set{\abs{b_1},\dots, \abs{b_t}},
			\]
			and
			\[
			h_i \geq \max\set{ h(\gamma_i), \abs{\log\gamma_i}D^{-1}, 0.16D^{-1} } \text{ for } 1 \leq i \leq t.
			\]    
		\end{theorem}
	
	\section{Approximating the roots}
		We start by giving good approximations for the roots of
		\begin{equation}\label{eq:f_n def}
			f_n(X) = X(X-A_n)(X-B_n) - 1 = \br{x-\lambda_1}\br{x-\lambda_2}\br{x-\lambda_3}.
		\end{equation}
		The polynomial is irreducible over $\Q$ as per the rational root theorem, at least for sufficiently large $n$: The only candidates for a rational root are $1$ and $-1$. When we plug them in and look at either $f_n(1) = 0$ or $f_n(-1) = 0$, we get that $A_n\cdot B_n = A_n + B_n$, which can hold at most for small $n$. \par\medskip
		
		To describe the quality of the approximation, we use the $L$-notation, where for two complex functions $f,g$, we write $f = L(g)$ if $\abs{f(x)}\leq \abs{g(x)}$ for all $x \in \mathbb{C}$. This way, we can consider constants that we would lose in the more classical $O$-notation.
		
		\begin{lemma}\label{lem:root-approx}
			We have, for sufficiently large $n$:
			\begin{align*}
				\lambda_1 &= B_n + L\br{B_n^{-1}} \\
				\lambda_2 &= A_n + L\br{A_n^{-1}} = A_n + \frac{1}{ A_n\br{A_n-B_n} } + L\br{ \frac{1}{ A_n^2\br{A_n-B_n}^2 } } \\
				\lambda_3 &= \frac{1}{A_n B_n} + L\br{ \frac{1}{A_n^2B_n^2} }
			\end{align*}
		\end{lemma}
		\begin{proof}
			We start with the approximation for $\lambda_1$. To that end, we first note that $f_n(B_n) = -1 < 0$. We then take a look at
			\[
			f_n(B_n \pm B_n^{-1}) = \br{B_n \pm B_n^{-1}} \br{ B_n-A_n \pm B_n^{-1} } \br{\pm B_n^{-1}} - 1,
			\]
			where we plug in $B_n + B_n^{-1}$ in case $1 \leq A_n \leq B_n - 2$ and $B_n-B_n^{-1}$ in the other. \par\medskip
			
			Now, since either $B_n-A_n \geq 2$ or $B_n - A_n \leq -3 $, we have either
			\[
			f_n(B_n + B_n^{-1}) \geq 1 + B_n^{-1} \br{ 1 + 2B_n^{-1} + B_n^{-2} } > 0
			\]
			or
			\[
			f_n(B_n - B_n^{-1}) \geq 2 - B_n^{-1} \br{ 1 + 2B_n^{-1} - B_n^{-2} } > 0
			\]
			for sufficiently large $n$. By the intermediate value theorem, a root $\lambda_1$ must lie between $B_n - B_n^{-1}$ and $B_n + B_n^{-1}$, thus $\lambda_1 = B_n + L\br{B_n^{-1}}$. \par\medskip
			
			We prove the first approximation for $\lambda_2$ analogously by showing that either $f_n(A_n - A_n^{-1}) > 0$ o r $f_n(A_n + A_n^{-1}) > 0$, depending on the sign of $A_n$. \par\medskip
			
			The second approximation for $\lambda_2$ follows analogously to the approximation for $\lambda_3$: We first write
			\[
			f_n\br{ \frac{1}{A_n B_n} + \kappa } = \br{ \frac{1}{A_n B_n} + \kappa }\br{ \frac{1}{A_n B_n} - A_n + \kappa }\br{ \frac{1}{A_n B_n} - B_n + \kappa } - 1
			\]
			and then expand the product and group by powers of $\kappa$. This then becomes
			\begin{align*}
				f_n\br{ \frac{1}{A_n B_n} + \kappa } = &\kappa \br{ A_n B_n - \frac{2}{A_n} - \frac{2}{B_n} + \frac{3}{A_n^2 B_n^2} } \\
				+ &\kappa^2\br{ -\br{ A_n+B_n } + \frac{3}{A_n B_n} } \\
				+ &\kappa^3 \\
				- & \br{ \frac{1}{A_n^2 B_n} + \frac{1}{A_n B_n^2} } + \frac{1}{A_n^3 B_n^3}.
			\end{align*}
			From this expansion, we see that the highest-order term is $\kappa A_n B_n$, as long as $\kappa$ is of order strictly greater than $\frac{1}{A_n^2B_n^2}\max\set{\frac{1}{A_n}, \frac{1}{B_n}}$. Thus, choosing $\kappa = \pm \frac{1}{A_n^2 B_n^2}$ gives alternating signs for $f_n\br{ \frac{1}{A_n B_n} + \kappa }$ and thus the claimed form for $\lambda_3$ by the intermediate value theorem.
		\end{proof}
		
		Their logarithms play as important a role as the roots themselves. We thus explicitly state their approximations too.
		
		\begin{lemma}\label{lem:log-root-approx}
			We have, up to an exponentially decreasing term $L\br{C\, n^{d_2} \varepsilon^n}$,
			\begin{align*}
				\log\abs{\lambda_1} &= n\log\abs{\beta} + \log\abs{c_B} \\
				\log\abs{\lambda_1-A_n} &=
				\begin{cases}
					n\log\abs{\beta} + \log\abs{c_B} \;\; &\text{ if } \abs{\alpha} < \abs{\beta} \\
					n \log\abs{\beta} + \log\abs{c_B-c_A} \;\; &\text{ if } \abs{\alpha} = \abs{\beta}
				\end{cases}\\ \\
				\log\abs{\lambda_2} &= n\log\abs{\alpha} + \log\abs{c_A} \\
				\log\abs{\lambda_2 - A_n} &= 
				\begin{cases}
					-n\br{\log\abs{\alpha}+\log\abs{\beta}} - \log\abs{c_A} - \log\abs{c_B} \;\; &\text{ if } \abs{\alpha} < \abs{\beta} \\
					-n\br{\log\abs{\alpha}+\log\abs{\beta}} - \log\abs{c_A} - \log\abs{c_B-c_A} \;\; &\text{ if } \abs{\alpha} = \abs{\beta} 
				\end{cases} \\ \\
				\log\abs{\lambda_3} &= -n\br{\log\abs{\alpha}+\log\abs{\beta}} - \log\abs{c_A} - \log\abs{c_B} \\
				\log\abs{\lambda_3-A_n} &= n\log\abs{\alpha} + \log\abs{c_A}
			\end{align*}
		\end{lemma}
		\begin{proof}
			The claims follow immediately from the approximations for the roots by Lemma \ref{lem:root-approx}. We demonstrate for $\lambda_1$:
			\[
			\log\abs{\lambda_1} = \log\abs{B_n + L\br{B_n^{-1}}} = \log\abs{B_n} + \log\abs{1 + L\br{B_n^{-2}}}.
			\]
			For sufficiently large $n$, we have $\abs{B_n}^{-2} \leq \frac{1}{2}$ and thus, by Lemma \ref{lem:log-ub},
			\[
			\log\abs{\lambda_1} = \log\abs{B_n} + L\br{2\abs{B_n^{-2}}}.
			\]
			We expand the term $\log\abs{B_n}$ next:
			\begin{align*}
				\log\abs{B_n} &= \log\abs{c_B \beta^n + c_{B,1} \beta_1^n + \cdots } \\
				&= n\log\abs{\beta} + \log\abs{c_B} + \log\abs{1 + \sum_i \frac{c_{B,i}}{c_B} \br{\frac{\beta_i}{\beta}}^n }.
			\end{align*}  
			The $c_B,c_{B,i}$ are polynomials in $n$ of degree at most $d_2$, as defined in Equation \eqref{def:min-max-degrees}. The first terms in the sum are thus of order at most $n^{d_2}$. All the while, the second terms decrease exponentially because of $\abs{\beta} > \max_i\abs{\beta_i}$. Thus, the sum will have an absolute value of at most $\frac{1}{2}$ for sufficiently large $n$, and we can use Lemma \ref{lem:log-ub} again. \par\medskip
			
			We thus have proved the form of $\log\abs{\lambda_1}$, up to an error of order
			\[
			L\br{ 2 \sum_i \abs{\frac{c_{B,i}}{c_B}} \abs{\frac{\beta_i}{\beta}}^n } + L\br{2 \abs{B_n^{-2}}}.
			\]
			The first $L$-term is of strictly higher order. We can thus ignore the second $L$-term if we take the first one $3$ times instead. We then bound the rational functions $\abs{\frac{c_{B,i}}{c_B}} \leq c_1 n^{d_2}$, where $c_1$ depends only on the coefficients of the $c_B, c_{B,i}$. We also have $\abs{\frac{\beta_i}{\beta}}^n \leq \max_i \abs{\frac{\beta_i}{\beta}}^n $ trivially and thus arrive at the $L$-term
			\[
			L\br{n^{d_2} \max_i \abs{\frac{\beta_i}{\beta}}^n \cdot 3c_1 \sum_i 1} = L\br{ 3c_1m_B \cdot n^{d_2} \max_i \abs{\frac{\beta_i}{\beta}}^n },
			\]
			which decreases exponentially. \par\medskip
			
			We analogously get
			\begin{align*}
				\log\abs{\lambda_2} &= \log\abs{A_n} + L\br{2 \abs{A_n}^{-2}}\\
				&= n\log\abs{\alpha} + \log\abs{c_A} + L\br{ 3c_2 m_A \cdot n^{d_2} \max_i \abs{\frac{\alpha_i}{\alpha}}^n },
			\end{align*}
			where we bound $\abs{\frac{c_{A,i}}{c_A}} \leq c_2 n^{d_2}$. \par\medskip
			
			If we look at $\lambda_1-A_n$, we get
			\[
			\log\abs{\lambda_1-A_n} = \log\abs{B_n-A_n} + L\br{2 \abs{B_n}^{-2}}
			\]
			and must now differentiate between the cases $\abs{\alpha} < \abs{\beta}$ and $\abs{\alpha} = \abs{\beta}$. In the first case, we factorise $\abs{c_B} \abs{\beta}^n$ and get the subsequent $L$-term
			\[
			3c_3 \br{m_B+m_A+1} n^{d_2} \max_i\set{\abs{\frac{\beta_i}{\beta}}, \abs{\frac{\alpha_i}{\beta}}, \abs{\frac{\alpha}{\beta}} }^n,
			\]
			where we bound $ \abs{\max_i\set{ \frac{c_A}{c_B}, \frac{c_{A,i}}{c_B}, \frac{c_{B_i}}{c_B} }} \leq c_3 n^{d_2}$. With $\abs{\alpha}<\abs{\beta}$, the term still decreases exponentially. In the second case, we substitute $\abs{\alpha}$ for $\abs{\beta}$ in $A_n$ and factorise $\abs{ c_B - c_A } \abs{\beta}^n$ instead. This then leads to the $L$-term
			\[
			3c_4 \br{m_B+m_A} n^{d_2} \max_i\set{\abs{\frac{\beta_i}{\beta}}, \abs{\frac{\alpha_i}{\beta}}}^n,
			\]
			with $ \abs{\max_i\set{ \frac{c_A}{c_B-c_A}, \frac{c_{A,i}}{c_B-c_A}, \frac{c_{B_i}}{c_B-c_A} }} \leq c_4 n^{d_2}$. \par\medskip
			
			The remaining logarithms follow analogously, as we already know the $L$-terms that result from $\log\abs{B_n}, \log\abs{A_n}, \log\abs{B_n-A_n}$. For example,
			\begin{align*}
				\log\abs{\lambda_2 - A_n} &= \log\abs{ A_n^{-1} \br{A_n-B_n}^{-1} + L\br{ A_n^{-2} \br{A_n-B_n}^{-2} } } \\
				&= -\log\abs{A_n} - \log\abs{B_n-A_n} + L\br{ \abs{A_n}^{-3} \abs{B_n-A_n}^{-3} },
			\end{align*} 
			and we can simply sum the $L$-terms that result from $\log\abs{A_n}$ and $\log\abs{B_n-A_n}$ and increase the constant by $1$ to account for the extra $L$-term. \par\medskip
			
			We can use the same $L$-term for all approximations by putting
			\begin{align*}
				C &= 5 \max\set{c_1,c_2,c_3,c_4} \br{m_B+m_A+1}, \\
				\varepsilon &= \max_i\set{ \abs{\frac{\beta_i}{\beta}}, \abs{\frac{\alpha_i}{\beta}}, \abs{\frac{\alpha_i}{\alpha}}, \chi_{\R \backslash \set{\abs{\beta}}}(\abs{\alpha}) \abs{\frac{\alpha}{\beta}} },
			\end{align*}
			where $\chi_{\R \backslash \set{\abs{\beta}}}(\abs{\alpha}) = 0$ if $\abs{\alpha}=\abs{\beta}$ and $1$ otherwise. This way, we can use the same exponentially decreasing $L$-term $L\br{ C\, n^{d_2} \varepsilon^n }$ across all approximations.
		\end{proof}
	
		\begin{lemma}\label{lem:root-diff}
			There are effectively computable constants $c_5,c_6 > 0$ such that the difference between two distinct roots $\lambda_i, \lambda_{i'}$ is bounded by
			\begin{align*}
				c_5 \abs{\beta}^n \leq \abs{\lambda_1-\lambda_2} \leq c_6\, n^{d_2} \abs{\beta}^n \\
				c_5\, n^{d_1} \abs{\beta}^n \leq \abs{\lambda_1-\lambda_3} \leq c_6\, n^{d_2} \abs{\beta}^n \\
				c_5\, n^{d_1} \abs{\alpha}^n \leq \abs{\lambda_2-\lambda_3} \leq c_6\, n^{d_2} \abs{\alpha}^n   
			\end{align*}
		\end{lemma}
		\begin{proof}
			By Lemma \ref{lem:root-approx}, the difference between two distinct roots is approximately either
			\begin{align*}
				\abs{\lambda_1 - \lambda_2} &\approx \abs{B_n-A_n} &&\approx \abs{c_B - \chi_{\set{\beta}}(\alpha) c_A} \abs{\beta}^n \\
				\abs{\lambda_1 - \lambda_3} &\approx \abs{B_n - \frac{1}{A_nB_n}} &&\approx \abs{c_B}\abs{\beta}^n \\
				\abs{\lambda_2 - \lambda_3} &\approx \abs{A_n - \frac{1}{A_nB_n}} &&\approx \abs{c_A}\abs{\alpha}^n,
			\end{align*}
			in the sense that we can bound it from below and above by the respective right-hand term, using effectively computable constants. Doing the same again, i.e. $n^{d_1} \ll \abs{c_A},\abs{c_B} \ll n^{d_2}$ and $1 \ll \abs{c_B - c_A} \ll n^{d_2}$, we can bound the difference between two distinct roots as in the statement.
		\end{proof}
	
	\section{Regulator and upper bound for $\log\abs{y}$}
		As per Theorem \ref{thm:Thomas1979}, the pairs $\lambda_i, \lambda_i-A_n$ are fundamental units of the order $\Z[\lambda_i]$. As they are generators, we can express the regulator $R$ up to sign as the determinant
		\[
		\pm R = \det
		\begin{pmatrix}
			\log\abs{\lambda_i} & \log\abs{\lambda_i-A_n} \\
			\log\abs{\lambda_{i'}} & \log\abs{\lambda_{i'} - A_n}
		\end{pmatrix},
		\]
		independent of the choice for $i,i'$.
		
		\begin{lemma}\label{lem:reg-square}
			For the regulator of the order $\Z[\lambda_i]$, we have $R = \Theta(n^2)$.
		\end{lemma}
		\begin{proof}
			We calculate the regulator using $(i,i') = (1,2)$, thus
			\[
			\pm R = \log\abs{\lambda_1}\log\abs{\lambda_2-A_n} - \log\abs{\lambda_2}\log\abs{\lambda_1-A_n}.
			\]
			With the approximations from Lemma $\ref{lem:log-root-approx}$, where we omit the $L$-terms for readability, this yields
			\begin{align*}
				\pm R &\approx \br{n\log\abs{\beta} + \log\abs{c_B}}\br{-n\br{\log\abs{\alpha}+\log\abs{\beta}}- \log\abs{c_A} - \log\abs{c_B}} \\
				&\;- \br{n\log\abs{\alpha}+\log\abs{c_A}} \br{ n\log\abs{\beta}+\log\abs{c_B} } \\
				&= -n^2 \, \log\abs{\beta}\br{ 2\log\abs{\alpha} + \log\abs{\beta} } + O\br{n \log n}
			\end{align*}
			if $\abs{\alpha} < \abs{\beta}$. If instead $\abs{\alpha} = \abs{\beta}$, we have
			\begin{align*}
				\pm &R \approx \br{n\log\abs{\alpha}+\log\abs{c_A}} \br{-n\br{\log\alpha+\log\abs{\beta}} - \log\abs{c_A} - \log\abs{c_B-c_A}} \\
				&\;- \br{n\log\abs{\alpha}+\log\abs{c_A}} \br{ n\log\abs{\beta} + \log\abs{c_B-c_A} } \\
				&= -n^2 \, \log\abs{\beta}\br{ 2\log\abs{\alpha} + \log\abs{\beta} } + O\br{n \log n}.
			\end{align*}
			In both cases, the regulator is quadratic in $n$.
		\end{proof}
	
		If we use the regulator $R$ being quadratic, which is then a quadratic upper bound to the regulator of the number field, we immediately get from Theorem \ref{thm:bugy1996}:
		
		\begin{lemma}\label{lem:logy-upper}
			We have
			\[
			\log\abs{y} \ll n^4\log n.
			\]
		\end{lemma}
		\begin{proof}
			Let $K = \Q(\lambda_1)$. By Lemma \ref{lem:reg-square}, we have $R_K \ll n^2$. The unit rank is $r=2$. The highest coefficient of the Thue equation is given by $A_n B_n$, which is at most of order $c_A \, c_B\, \beta^{2n}$. We plug everything into the upper bound from Theorem \ref{thm:bugy1996}, and as the highest-order term is of the form $n^4 \log n$, the claim follows.
		\end{proof}
	
	\section{Constructing the linear form in logarithms}
		We now set out to construct a linear form in logarithms, on which we want to use Baker's and W\"ustholz's lower bound to then deduce a lower bound for $\log\abs{y}$. To that end, we start with the expressions
		$x-\lambda_i y$ and fix the index $j$, sometimes called the type of the solution $(x,y)$, such that
		\[
		\abs{x-\lambda_j y} \leq \abs{x-\lambda_i y} \;\;\; \forall i \in \set{1,2,3},
		\]
		and denote the other two indices by $k$ and $l$, i.e. $\set{j,k,l} = \set{1,2,3}$. First, we ascertain that the $x-\lambda_i y$ give units in $\Z[\alpha_i]$.
		
		\begin{lemma}
			For each $i \in \set{1,2,3}$, we have $x-\lambda_i y \in \br{\Z[\lambda_i]}^\times$.
		\end{lemma}
		\begin{proof}
			We take the Norm $N = N_{\Q(\lambda_1) / \Q }$, which is the product of the conjugates, i.e.
			\[
			N(x-\lambda_i y) = \br{x-\lambda_j y}\br{x-\lambda_k y}\br{x-\lambda_l y}.
			\]
			Now the right hand side is $y^3f_n\br{\frac{x}{y}}$ and thus exactly the left hand side of our Thue equation at $(x,y)$. Thus, $\abs{N(x-\lambda_i y)} = 1$, from which the claim follows.
		\end{proof}
	
		As the $x-\lambda_i y$ are units in $\mathbb{Z}[\lambda_i]$, we can write them in terms of the fundamental units $\lambda_i, \lambda_i - A_n$. Let thus
		\begin{equation}\label{eq:unitdecomp}
			x-\lambda_i y = \pm \lambda_i^{b_1} \br{\lambda_i-A_n}^{b_2}
		\end{equation}
		for each $i \in \set{1,2,3}$. We then take a look at Siegel's identity,
		\[
		\br{x-\lambda_j \, y} \br{ \lambda_k - \lambda_l } + \br{x- \lambda_l \, y} \br{ \lambda_j - \lambda_k } + \br{x-\lambda_k \, y} \br{ \lambda_l - \lambda_j } = 0,
		\]
		which we rewrite into
		\[
		\frac{x-\lambda_l y}{x-\lambda_k y} \cdot \frac{\lambda_j-\lambda_k}{\lambda_j-\lambda_l} = 1 + \frac{x - \lambda_j y}{x - \lambda_k y} \cdot \frac{\lambda_l - \lambda_k}{\lambda_j - \lambda_l} =: 1 + \gamma.
		\]
		We take the absolute value and use our decomposition into fundamental units on the first term of the left hand side, thus defining the linear form in logarithms
		\[
			\Lambda :=  b_1 \log\abs{ \frac{\lambda_l}{\lambda_k} } + b_2 \log\abs{ \frac{\lambda_l-A_n}{\lambda_k-A_n} } + \log\abs{ \frac{\lambda_j-\lambda_k}{\lambda_j-\lambda_l} }  = \log\abs{1+\gamma},
		\]
		which is very small, as $\gamma$ is very small. In fact:
		\begin{lemma}\label{lem:Lambda-ub}
			We have
			\[
			\abs{\gamma} \leq 2c_5^3\, n^{-d_1} \abs{\alpha}^{-2n} \abs{\beta}^{-n}
			\]
			and, if $2c_5^3\, n^{-d_1} \abs{\alpha}^{-2n} \abs{\beta}^{-n} \leq \frac{1}{2}$,
			\[
			\abs{\Lambda} \leq 4c_5^3\, n^{-d_1} \abs{\alpha}^{-2n} \abs{\beta}^{-n}.
			\]
		\end{lemma}
		\begin{proof}
			As
			\[
			\gamma = \frac{x - \lambda_j y}{x - \lambda_k y} \cdot \frac{\lambda_l - \lambda_k}{\lambda_j - \lambda_l},
			\]
			we have to bound all four terms. For $i \neq j$, we have 
			\[
			\abs{y} \abs{\lambda_i-\lambda_j} \leq \abs{x-\lambda_i y} + \abs{x-\lambda_j y} \leq 2 \abs{x-\lambda_i y}.
			\]
			In combination with $\abs{y} \geq 2$, we thus have
			\[
			\abs{x-\lambda_k y } \geq  \abs{\lambda_k - \lambda_j}
			\]
			and, as the product $\br{x-\lambda_j y}\br{x-\lambda_k y}\br{x-\lambda_l y} = \pm 1$,
			\[
			\abs{x-\lambda_j y} = \abs{x-\lambda_k y}^{-1} \abs{x - \lambda_l y}^{-1} \leq  \abs{\lambda_k-\lambda_j}^{-1}\abs{\lambda_l - \lambda_j}^{-1}.
			\]
			We put everything together and get
			\begin{align*}
				\abs{\frac{x-\lambda_j y}{x-\lambda_k y}} \cdot \abs{ \frac{\lambda_l - \lambda_k}{\lambda_j - \lambda_l} } \leq \frac{\abs{\lambda_k-\lambda_j}^{-1}\abs{\lambda_l-\lambda_j}^{-1}}{\abs{\lambda_k-\lambda_j}} \cdot \frac{ \abs{\lambda_l - \lambda_j} + \abs{\lambda_k - \lambda_j} }{\abs{\lambda_l - \lambda_j}},
			\end{align*}
			where we also used the triangle inequality on the numerator of the second fraction. The right hand side is now equal to
			\[
			\abs{\lambda_k-\lambda_j}^{-2}\abs{\lambda_l-\lambda_j}^{-1} + \abs{\lambda_k-\lambda_j}^{-1}\abs{\lambda_l-\lambda_j}^{-2} \leq 2c_5^3\, n^{-d_1} \abs{\alpha}^{-2n} \abs{\beta}^{-n}
			\]
			as per Lemma \ref{lem:root-diff}, and the bound decreases exponentially even if $\abs{\alpha} = 1$.
		\end{proof}
		
		We'd like to use the lower bound from Theorem \ref{thm:baker1993} directly on $\Lambda$. However, as the arguments of some of the logarithms are exponential in $n$, their logarithmic heights are linear in $n$. This will give $\log\abs{\Lambda} \gg -n^3  \log\br{ \frac{1}{n} \log\abs{y} }$, from which we can only deduce
		\begin{equation}\label{eq:logy-lb1}
			\log\abs{y} \gg n
		\end{equation}
		instead of the desired exponential lower bound. \par\medskip
		
		We thus aim to shift the exponential dependency on $n$ from the logarithms into the coefficients, where an additional factor $n$ will not ruin the argument.
		
	\subsection{Transforming the linear form}
		Let $j = 1$ and choose $(k,l) = (3,2)$. We use the approximations from Lemma \ref{lem:log-root-approx} to write:
		\begin{align*}
			\log\abs{\frac{\lambda_l}{\lambda_k}} &+ L\br{2C\, n^{d_2} \varepsilon^n} = \log\abs{\lambda_2} - \log\abs{\lambda_3} + L\br{2C\, n^{d_2} \varepsilon^n} \\
			&= n \br{2\log\abs{\alpha}+\log\abs{\beta}} + 2\log\abs{c_A} + \log\abs{c_B}\\ \\
			\log\abs{\frac{\lambda_l-A_n}{\lambda_k-A_n}} &+ L\br{2C\, n^{d_2} \varepsilon^n} = \log\abs{\lambda_2-A_n} - \log\abs{\lambda_3-A_n} + L\br{2C\, n^{d_2} \varepsilon^n} \\
			&= 
			\begin{cases}
				-n \, \br{2\log\abs{\alpha}+\log\abs{\beta}} - 2\log\abs{c_A} - \log\abs{c_B} \;\; &\text{ if } \abs{\alpha} < \abs{\beta} \\
				-n \, \br{2\log\abs{\alpha}+\log\abs{\beta}} - 2\log\abs{c_A} - \log\abs{c_B-c_A} \;\; &\text{ if } \abs{\alpha} = \abs{\beta}
			\end{cases}\\ \\
			\log\abs{\frac{\lambda_k-\lambda_j}{\lambda_l-\lambda_j}} &+ L\br{4C\, n^{d_2} \varepsilon^n}  = \log\abs{\lambda_3-\lambda_1} - \log\abs{\lambda_2-\lambda_1} + L\br{4C\, n^{d_2} \varepsilon^n} \\
			&=
			\begin{cases}
				0 \;\; &\text{ if } \abs{\alpha} < \abs{\beta} \\
				\log\abs{c_B} - \log\abs{c_B-c_A} \;\; &\text{ if } \abs{\alpha} = \abs{\beta}
			\end{cases}
		\end{align*}
		We shift the $L$-terms into the upper bound, and get a linear form in at most $4$ logarithms. The coefficients are gathered in the table below, where $\chi_M$ denotes the characteristic function on the set $M$.
		
		\begin{table}[H]
			\begin{tabular}{l|l}
				& $j=1$  \\ \hline
				$\log\abs{\alpha}$ & $2n\br{b_1-b_2}$ \\
				$\log\abs{\beta}$ & $n\br{b_1-b_2}$ \\
				$\log\abs{c_A}$ & $2\br{b_1-b_2}$ \\
				$\log\abs{c_B}$ & $b_1 - \chi_{\R\backslash\set{\abs{\beta}}}\br{\abs{\alpha}} b_2 + \chi_{\set{\abs{\beta}}}\br{\abs{\alpha}}$ \\
				$\log\abs{c_B-c_A}$ & $\chi_{\set{\abs{\beta}}}\br{\abs{\alpha}} \br{b_2+1}$ 
			\end{tabular}
		\end{table}
		
		We do the same for the cases $(j,k,l) = (2,3,1)$ and $(j,k,l) = (3,2,1)$ and get:
		
		\begin{table}[H]
			\begin{tabular}{l|ll}
				& $j=2$ & $j=3$ \\ \hline
				$\log\abs{\alpha}$ & $n\br{b_1-\br{b_2-1}}$  & $n\br{b_2 - \br{b_1-1}}$ \\
				$\log\abs{\beta}$ & $n\br{2b_1+b_2-1}$ & $n\br{2b_2 + b_1-1}$ \\
				$\log\abs{c_A}$ & $b_1-\br{b_2-1}$ & $b_2 + -\br{b_1-1}$ \\
				$\log\abs{c_B}$ & $2b_1 + \chi_{\R\backslash\set{\abs{\beta}}}\br{\abs{\alpha}}\br{b_2-1}$ & $\chi_{\R\backslash\set{\abs{\beta}}}\br{\abs{\alpha}} 2b_2 - 1$ \\
				$\log\abs{c_B-c_A}$ & $\chi_{\set{\abs{\beta}}}\br{\abs{\alpha}}\br{b_2-1}$ & $\chi_{\set{\abs{\beta}}}\br{\abs{\alpha}} 2b_2$
			\end{tabular}
		\end{table}	
		
		We call the linear form in the logarithms and coefficients as per the tables above $\xi_j$. As we shifted the $L$-terms into the upper bound, we have to worsen it accordingly, we thus have
		\begin{equation}\label{eq:xi_j-ub}
			\abs{\xi_j} \leq 4c_5^3\, n^{-d_1} \abs{\alpha}^{-2n} \abs{\beta}^{-n} + 6C\, n^{d_2} \varepsilon^n \ll n^{d_2} \varepsilon^n.
		\end{equation}
	
	\subsection{A closer look at the powers}
		To apply Baker's and W\"ustholz's lower bound on the linear form $\xi$, we have to argue $\xi \neq 0$ first. For that, we have to either restrict the logarithms, such as demanding linear independence or gain further information on the coefficients, and thus on $b_1,b_2$. \par\medskip
		
		To that end, we return to the equations
		\[
			\abs{x-\lambda_i y} = \abs{\lambda_i^{b_1} \br{\lambda_i-A_n}^{b_2}} \;\; \forall i \in \set{1,2,3}.
		\]
		If we take the logarithm, the equations for $i = k,l$ become
		\begin{equation}\label{eq:powerLGS1}
			\begin{pmatrix}
				\log\abs{x-\lambda_k y} \\
				\log\abs{x-\lambda_l y}
			\end{pmatrix}
			=
			\begin{pmatrix}
				\log\abs{\lambda_k} & \log\abs{\lambda_k-A_n} \\
				\log\abs{\lambda_l} & \log\abs{\lambda_l-A_n}
			\end{pmatrix}
			\begin{pmatrix}
				b_1 \\
				b_2
			\end{pmatrix}.
		\end{equation}
		The $2\times 2$ matrix on the right hand side is invertible, as its determinant is equal to the regulator $R$, which is non-zero as per Lemma \ref{lem:reg-square}. We multiply the equation with the inverse matrix
		\[
			\frac{1}{R}
			\begin{pmatrix}
				\log\abs{\lambda_l-A_n} & -\log\abs{\lambda_k-A_n} \\
				-\log\abs{\lambda_l} & \log\abs{\lambda_k}
			\end{pmatrix}
		\]
		to get an explicit formula for $b_1,b_2$. We further break down the formula by writing
		\begin{align*}
			\log\abs{x-\lambda_i y} &= \log\abs{x-\lambda_j y - \br{\lambda_i-\lambda_j} y} \\
			&= \log\abs{y} + \log\abs{\lambda_i-\lambda_j} + \log\br{ 1 - \frac{x-\lambda_j y}{\br{\lambda_i-\lambda_j}y} }
		\end{align*}
		for $i = k,l$. The last term is small, as seen in the proof of Lemma $\ref{lem:Lambda-ub}$. By writing
		\[
			\frac{\abs{x-\lambda_j}}{\abs{\lambda_i-\lambda_j}\abs{y}} \leq \frac{\abs{\lambda_k-\lambda_j}^{-1}\abs{\lambda_l-\lambda_j}^{-1}}{2\abs{\lambda_i-\lambda_j}} \leq \frac{1}{2} c_5^3\, n^{-d_1} \abs{\alpha}^{-2n} \abs{\beta}^{-n}
			\]
			as per Lemma \ref{lem:root-diff}, we thus have
			\[
			\log\abs{x-\lambda_i y} = \log\abs{y} + \log\abs{\lambda_i-\lambda_j} + L\br{\frac{1}{2} c_5^3\, n^{-d_1} \abs{\alpha}^{-2n} \abs{\beta}^{-n}}.
		\] 
		
		We use this expression to rewrite Equation \eqref{eq:powerLGS1}, multiplied with the inverse matrix, into
		\begin{equation}\label{eq:powerLGS2}
			\begin{pmatrix}
				b_1 \\
				b_2
			\end{pmatrix}
			=
			\frac{1}{R}
			\begin{pmatrix}
				\log\abs{\lambda_l-A_n} & -\log\abs{\lambda_k-A_n} \\
				-\log\abs{\lambda_l} & \log\abs{\lambda_k}        
			\end{pmatrix}
			\begin{pmatrix}
				\log\abs{y} + \log\abs{\lambda_k-\lambda_j} \\
				\log\abs{y} + \log\abs{\lambda_l-\lambda_j}
			\end{pmatrix},
		\end{equation}
		up to an error term $L\br{C'(n)\abs{\alpha}^{-2n}\abs{\beta}^{-n}}$ with
		\[
		C'(n) := \frac{1}{R} \log\abs{\frac{\lambda_k \br{\lambda_l-A_n}}{\lambda_l \br{\lambda_k-A_n}}} c_5^3 n^{-d_1},
		\]
		and we mainly care that the $L$-term decreases exponentially and that $C'(n) \ll 1$. \par\medskip
		
		In the case that $j=1$ and $\abs{\alpha}<\abs{\beta}$ only, we can actually forego to fiddle around with the linear form and directly derive an exponential lower bound for $\log\abs{y}$. To that end, we use a different set of fundamental units, $\lambda_i$ and $\lambda_i - B_n$. They still fall under the results of Theorem \ref{thm:Thomas1979}, and Equation \eqref{eq:powerLGS2} holds for the new set of powers $u_1, u_2$ if we substitute $A_n$ for $B_n$ in the matrix.
		
		\begin{lemma}
			If $j = 1$ and $\abs{\alpha} < \abs{\beta}$, then $\log\log\abs{y} \gg n$.
		\end{lemma}
		\begin{proof}
			We take a closer look at $\pm R\,u_1$, which is
			\[
			\br{\log\abs{\lambda_2-B_n}-\log\abs{\lambda_3-B_n}} \log\abs{y} + t(n) + L\br{ C'(n)\abs{\alpha}^{-2n}\abs{\beta}^{-n} },
			\]
			where we only care that $\abs{t(n)}$ at most decreases exponentially per Lemma \ref{lem:log-root-approx}, so $\abs{t(n)} \ll c_7^{-n}$. \par\medskip
			
			We have $c_8^{-n} \ll \abs{ \log\abs{\lambda_2-B_n}-\log\abs{\lambda_3-B_n} } \ll c_9^{-n} $ with the same argument. By Equation \eqref{eq:logy-lb1} we also have $\log\abs{y} \gg n$, while $C'(n) \ll 1$. Thus, $\abs{Ru_1} \neq 0$, hence $\abs{Ru_1} \geq R$ for sufficiently large $n$. \par\medskip
			
			This, however, means that
			\[
			\log\abs{y} \geq \frac{R - t(n) - C'(n)\abs{\alpha}^{-2n}\abs{\beta}^{-n} }{ \abs{ \log\abs{\lambda_2-B_n}-\log\abs{\lambda_3-B_n} } } \gg c_{10}^n,
			\]
			and if $c_7,c_8,c_9 > 1$, so is $c_{10}$.
		\end{proof}
		
		\begin{lemma}\label{lem:xinonzero}
			For $j \in \set{2,3}$ or $j=1,\, \abs{\alpha} = \abs{\beta}$, we have $\xi_j \neq 0$.
		\end{lemma}
		\begin{proof}
			Assume $j=1$ and take $(k,l) = (3,2)$. By Equation \eqref{eq:powerLGS2} and Lemma \ref{lem:log-root-approx}, the powers $b_1, b_2$ are then, up to exponentially decreasing errors,
			\begin{align*}
				R\,b_1 = &\big(-n\br{\log\abs{\alpha}+\log\abs{\beta}} - \log\abs{c_A} - \chi_{\R\backslash\set{\abs{\beta}}}(\abs{\alpha})\log\abs{c_B} \\ &-\chi_{\set{\abs{\beta}}}(\abs{\alpha})\log\abs{c_B-c_A} \big)\br{ \log\abs{y} + n\log\abs{\beta} + \log\abs{c_B} } \\
				- &\br{n\log\abs{\alpha} + \log\abs{c_A}}\\
				&\br{\log\abs{y} + n\log\abs{\beta} + \chi_{\R\backslash\set{\abs{\beta}}}(\abs{\alpha})\log\abs{c_B} - \chi_{\set{\abs{\beta}}}(\abs{\alpha})\log\abs{c_B-c_A} }
			\end{align*}
			and
			\begin{align*}
				R\,b_2 = - &\br{n\log\abs{\alpha} + \log\abs{c_A}}\br{\log\abs{y} + n\log\abs{\beta} + \log\abs{c_B}} \\
				+ &\br{-n\br{\log\abs{\alpha}+\log\abs{\beta}}-\log\abs{c_A}-\log\abs{c_B}}\\
				&\br{\log\abs{y}+\log\abs{\beta}+\chi_{\R\backslash\set{\abs{\beta}}}(\abs{\alpha})\log\abs{c_B} - \chi_{\set{\abs{\beta}}}(\abs{\alpha})\log\abs{c_B-c_A}}.
			\end{align*}
			
			If we multiply the terms out and look at the highest order terms -- those of $n^2$ and $\log\abs{y}\cdot n$, then we see that they have negative sign and are the same for $b_1$ and $b_2$. The other terms all get dominated by the factor $\frac{1}{R}$, which decreases like $\frac{1}{n^2}$, as per Lemma \ref{lem:reg-square}. Thus, $b_1 = b_2$ for sufficiently large $n$. Note that this would mean that all coefficients in $\xi_1$ are $0$ if $\abs{\alpha}<\abs{\beta}$. \par\medskip
			
			If $\abs{\alpha}=\abs{\beta}$ however, $b_1 = b_2$ immediately implies that $\xi_1 \neq 0$, except for when $b_1 = b_2 = -1$, as per our conditions on $\abs{c_B}, \abs{c_B-c_A}$ in Theorem \ref{thm:main}. \par\medskip
			
			Going back to the unit decompositions in Equation \eqref{eq:unitdecomp}, we can rewrite them into
			\[
				\br{x-\lambda_i y}\lambda_i^{-b_1}\br{\lambda_i-A_n}^{-b_2} - 1 = 0
			\]
			for $i\in\set{1,2,3}$. With $b_1 = b_2 = -1$, this is a polynomial equation and the $\lambda_i$ thus roots of the polynomial 
			\[
				\br{-yX+x}X\br{X-A_n}-1,
			\]
			which must then be divisible by the minimal polynomial $X\br{X-A_n}\br{X-B_n}-1$. The constant terms are equal, so the leading coefficients must be, too. This gives $y = -1$ and thus the contradiction to $(x,y)$ being a non-trivial solution. So $j$ cannot be $1$ if $\abs{\alpha} = \abs{\beta}$. \\
			
			Assume, then, $j = 2$ and take $(k,l) = (3,1)$. Equation \eqref{eq:powerLGS2} and Lemma \ref{lem:log-root-approx} give
			\begin{align*}
				R\,b_1 = &\br{n \log\abs{\beta} + \chi_{\R\backslash\set{\abs{\beta}}}(\abs{\alpha})\log\abs{c_B} + \chi_{\set{\abs{\beta}}}(\abs{\alpha})\log\abs{c_B-c_A} } \\
				&\br{\log\abs{y}+n\log\abs{\alpha}+\log\abs{c_A}} \\
				- &\br{n\log\abs{\alpha} + \log\abs{c_A}} \\
				&\br{\log\abs{y} + n\log\abs{\beta} + \chi_{\R\backslash\set{\abs{\beta}}}(\abs{\alpha})\log\abs{c_B} + \chi_{\set{\abs{\beta}}}(\abs{\alpha})\log\abs{c_B-c_A} }
			\end{align*}
			and
			\begin{align*}
				R\, b_2 = -&\br{n\log\abs{\beta}+\log\abs{c_B}}\br{\log\abs{y}+n\log\abs{\alpha}+\log\abs{c_A}}\\
				+&\br{-n\br{\log\abs{\alpha}+\log\abs{\beta}}-\log\abs{c_A}-\log\abs{c_B}}\\
				&\br{\log\abs{y}+n\log\abs{\beta} + \chi_{\R\backslash\set{\abs{\beta}}}(\abs{\alpha})\log\abs{c_B} + \chi_{\set{\abs{\beta}}}(\abs{\alpha})\log\abs{c_B-c_A}}
			\end{align*}
			up to exponentially decreasing errors. Further assume that $\abs{\alpha} \neq \abs{\beta}$ and we look at the logarithms $\log\abs{\alpha}, \log\abs{\beta}$ in $\xi_2$. The term in question is
			\[
				n\br{b_1 - \br{b_2-1}} \log\abs{\alpha} + n\br{2b_1 + b_2 - 1}\log\abs{\beta}
			\]
			and dominates the other terms. Thus, it suffices to show that its highest order terms do not cancel each other out, which is easily verified with the formula derived for $b_1$ and $b_2$. \par\medskip
			
			If $\abs{\alpha} = \abs{\beta}$ instead, the same term simplifies to $3nb_1 \log\abs{\alpha}$ and the formula for $b_1$ itself to $R\, b_1 = \log\abs{y} \br{\log\abs{c_B-c_A} - \log\abs{c_A}}$. We can no longer claim that $3nb_1$ dominates the other coefficients of $\xi_2$ and have to take a closer look. \par\medskip
			
			To that end, we first write $R\, b_2$ explicitly in terms of $n^2, n$ and $\log\abs{y}$. We have
			\begin{align*}
				R\,b_2 &= n^2\br{-3\br{\log\abs{\alpha}}^2} + n \cdot 2\log\abs{\alpha}\br{-\log\abs{c_A}-\log\abs{c_B}-\log\abs{c_B-c_A}} \\
				&+ \log\abs{y}\br{-3n\log\abs{\alpha} - 2\log\abs{c_B} - \log\abs{c_A}}.
			\end{align*}
			
			We plug everything into the linear form $\xi_2$,
			\[
				3nb_1\log\abs{\alpha} + \br{b_1-b_2+1} \log\abs{c_A} + 2b_1 \log\abs{c_B} + \br{b_2-1}\log\abs{c_B-c_A},
			\]
			and group the terms with respect to their order in $n$, differentiating two cases. Either the term with $\frac{1}{R}\log\abs{y}\,n$, the whole term is 
			\[
				\frac{1}{R}\log\abs{y}\,n \cdot 3\log\abs{\alpha}\br{\log\abs{c_A}-\log\abs{c_B-c_A}},
			\]
			is at most constant, equivalently $\log\abs{y} = O(n)$. In that case, since $\log\abs{c_A}-\log\abs{c_B-c_A} > 0$ as per our conditions on $A_n,B_n$, the whole term equals at most some positive constant. While the term $\frac{1}{R} n^2 \cdot 3\log\abs{\alpha}^2\br{\log\abs{c_A}-\log\abs{c_B-c_A}}$ equals some positive constant, and thus their sum. All the other terms decrease with increasing $n$, thus the linear form $\xi_2$ does not cancel for sufficiently large $n$. In the other case, the term with $\frac{1}{R}\log\abs{y}\cdot n$ dominates all others -- including the one with $\frac{1}{R}n^2$ -- and we can conclude $\xi_2 \neq 0$ for sufficiently large $n$ just by the non-vanishing of the $\frac{1}{R}\log\abs{y}\, n$ term. \\
			
			Finally, assume $j = 3$ and take $(k,l) = (2,1)$. We again have, by Equation \eqref{eq:powerLGS2} and Lemma \ref{lem:log-root-approx},
			\begin{align*}
				R\, b_1 = &\br{n \log\abs{\beta} + \chi_{\R\backslash\set{\abs{\beta}}}(\abs{\alpha})\log\abs{c_B} + \chi_{\set{\abs{\beta}}}(\abs{\alpha})\log\abs{c_B-c_A}} \\
				&\br{\log\abs{y}+n\log\abs{\alpha}+\log\abs{c_A}} \\
				- &\big(-n\br{\log\abs{\alpha}+\log\abs{\beta}}-\log\abs{c_A}- \chi_{\R\backslash\set{\abs{\beta}}}(\abs{\alpha})\log\abs{c_B} \\ &+\chi_{\set{\abs{\beta}}}(\abs{\alpha})\log\abs{c_B-c_A} \big) \br{\log\abs{y}+n\log\abs{\beta}+\log\abs{c_B}}
			\end{align*}
			and
			\begin{align*}
				R\,b_2 = -&\br{n\log\abs{\beta}+\log\abs{c_B}}\br{\log\abs{y}+n\log\abs{\alpha}+\log\abs{c_A}}\\
				+&\br{n\log\abs{\alpha}+\log\abs{c_A}}\br{\log\abs{y}+n\log\abs{\beta}+\log\abs{c_B}}
			\end{align*}
			up to exponentially decreasing errors. Again, if $\abs{\alpha}\neq\abs{\beta}$, the coefficients of $\log\abs{\alpha},\log\abs{\beta}$ dominate the others in $\xi_3$, the term in question being
			\[
				n\br{b_2-\br{b_1-1}}\log\abs{\alpha} + n\br{2b_2 + \br{b_1-1}}\log\abs{\beta}.
			\]
			It follows from the formulas for $b_1, b_2$ that its highest order terms do not cancel each other out and we can thus conclude $\xi_3\neq 0$. \par\medskip
			
			If $\abs{\alpha}=\abs{\beta}$, this again fails, as the term simplifies to $3nb_2\log\abs{\alpha}$ and $b_2$ to $R\, b_2 = \log\abs{y}\br{\log\abs{c_A}-\log\abs{c_B}}$. We state $R\, b_1$ in terms of $n^2$, $n$, and $\log\abs{y}$ and get
			\begin{align*}
				R\, b_1 &= n^2\br{3\br{\log\abs{\alpha}}^2} + n \cdot 2\log\abs{\alpha}\br{\log\abs{c_A}+\log\abs{c_B}+\log\abs{c_B-c_A}} \\
				&+ \log\abs{y}\br{3n\log\abs{\alpha} + \log\abs{c_A}+2\log\abs{c_B-c_A}}.
			\end{align*}
			We then plug everything into the linear form $\xi_3$,
			\[
				3nb_2\log\abs{\alpha} + \br{b_2+b_1-1}\log\abs{c_A} - \log\abs{c_B} + 2b_2 \log\abs{c_B-c_A},
			\]
			and conclude analogously to the case $j=2$ that $\xi_3 \neq 0$. It should be noted that $\log\abs{c_A}-\log\abs{c_B}$ is non-zero, but could be either positive or negative. However, the factor occurs in both the terms with $\frac{1}{R}n^2$ and $\frac{1}{R}\log\abs{y}\,n$ and is the only one of questionable sign, thus both terms have the same sign and do not cancel with the same argument as for the case $j=2$.
		\end{proof}
	
		With $\xi_j \neq 0$, we can apply Theorem \ref{thm:baker1993} and derive a lower bound for $\log\abs{\xi}$. The coefficients of $\xi_j$ can be bounded by $n \cdot \max\set{b_1,b_2}$. By Equation \eqref{eq:powerLGS2}, in combination with $R$ being quadratic by Lemma \ref{lem:reg-square} and $\log\abs{y}$ being at most of order $n^4\log n$ by Lemma \ref{lem:logy-upper}, we have $B \ll n^4\log n$ in Theorem \ref{thm:baker1993}. The terms $c_A,c_B$ are at most polynomial in $n$, thus their logarithmic heights are of order at most $\log n$, i.e. $h\br{\abs{c_A}},h\br{\abs{c_B}},h\br{\abs{c_B-c_A}} \ll \log n$, and the logarithmic heights of $\abs{\alpha}, \abs{\beta}$ are constants. \par\medskip
		
		Everything put together, Theorem \ref{thm:baker1993} gives
		\begin{equation*}
			\log\abs{\xi_j} \gg - \br{\log n}^3 \cdot \log\br{ n^4 \log n }.
		\end{equation*}
		If we take the logarithm of the upper bound from Equation \eqref{eq:xi_j-ub}, we get
		\[
		-n \gg -\br{\log n}^3 \cdot \log\br{n^4\log n},
		\]
		which gives us a contradiction for sufficiently large $n$, or in turn an effective upper bound for $n$. We thus conclude our proof of Theorem \ref{thm:main}.
		
		\par\medskip
		\centerline{\bf Acknowledgement}
		\noindent
		The author wants to thank Volker Ziegler for the initial idea as well as many helpful discussions and feedback.
		

\end{document}